\documentclass[12pt,leqno]{article}
\usepackage{amsmath,amssymb,amsbsy,amsfonts,amsthm,latexsym,
         amsopn,amstext,amsxtra,euscript,amscd,amsthm,mathdots}

\allowdisplaybreaks

\numberwithin{equation}{section}

\newtheorem{theorem}{Theorem}[section]
\newtheorem{corollary}[theorem]{Corollary}
\newtheorem{lemma}[theorem]{Lemma}

\theoremstyle{definition}

\newtheorem{remark}[theorem]{Remark}

\newcommand{\sgn}[1]{\operatorname{\rm{sgn}}(#1)}

\begin{document}

\begin{large}
\centerline{\large \bf Truncated convolution of M\"obius function}
\centerline{\large \bf and multiplicative energy of an integer $n$}
\end{large}
\vskip 10pt
\begin{large}
\centerline{\sc  Patrick Letendre }
\end{large}
\vskip 10pt
\begin{abstract}
We establish an interesting upper bound for the moments of truncated Dirichlet convolution of M\"obius function, a function noted $M(n,z)$. Our result implies that $M(n,j)$ is usually quite small for $j \in \{1,\dots,n\}$. Also, we establish an estimate for the multiplicative energy of the set of divisors of an integer $n$.
\end{abstract}


\vskip 20pt

\section{Introduction}

Let $\mu(\cdot)$ be the M\"obius function and consider
$$
M(n,z) := \sum_{\substack{d \mid n \\ d  \le z}} \mu(d).
$$
The function $M(n,z)$ has been studied by various authors (see \cite{pe}, \cite{ngdb:cvet:dk}, \cite{pe:ik}, \cite{rrh}, \cite{pe:rrh}, \cite{hm}, \cite{rdlb:gt} for example). In \cite{pe} it is established that
\begin{equation}\label{lges}
\biggl|\sum_{\substack{d \mid n \\ a \le d \le b}} \mu(d)\biggr| \le \binom{\omega(n)}{\lfloor\frac{\omega(n)}{2}\rfloor} \quad (1 \le a \le b \le n)
\end{equation}
where $\omega(n)$ is the number of distinct prime divisors of $n$. A very interesting tool, known as the {\it symmetrical chains}, is used to establish a generalization of this property in \cite{ngdb:cvet:dk}. In this paper, we are interested by the average size of $M(n,z)$ over $1 \le z \le n$. More precisely, we consider the quantity
\begin{equation}\label{Lt_def}
L_t(n):=\int_{1}^{n}M(n,z)^tdz
\end{equation}
for integer values of $t \ge 1$. Let's remark that $L_t(n)=L_t(\gamma(n))$, where $\gamma(n):=\prod_{p \mid n}p$. From what we know, only the value of $L_1(n)$, which is $-\prod_{p \mid n}(1-p)$ for $n \ge 2$, is easy to evaluate. Let's write $\log_+ x:=\log \max(x,2)$ and $\delta_{i,j}$ for the Kronecker delta.

\begin{theorem}\label{thm:1}
Let $t \ge 2$ be an integer and $n \ge 1$ be a squarefree integer. Then
$$
|L_t(n)| \le (1+\delta_{2,t})n\exp\left( Ct\frac{\omega(n)^{1-\frac{1}{t}}}{(1-1/t)\log^{\frac{1}{t}}_+ \omega(n)}\right),
$$
where $C=1.07073472\dots$ is defined in the statement of Lemma \ref{comp}.
\end{theorem}

\begin{theorem}\label{thm:2}
Let $t \ge 2$ be an integer and $n \ge 1$ be a squarefree integer. Then
$$
|L_t(n)| \le \left\{\begin{array}{ll} 2n\exp\bigl( t\omega(n)^{1-\frac{1}{t}}\bigr) & \mbox{if}\ t = 2\ \mbox{and}\ \omega(n) \le 55,\\ [2mm] n\exp\bigl( t\omega(n)^{1-\frac{1}{t}}\bigr) & \mbox{otherwise}.\end{array}\right.
$$
\end{theorem}

We then use Theorem 2 to get some control over the quantity
\begin{equation}\label{G_def}
\mathcal{H}_{\theta}(n):=|\{j \in [1,n] \cap \mathbb{N}:\ |M(n,j)| \ge 2^{\theta\omega(n)}\}| \quad (\theta \in (0,1]).
\end{equation}
To express our result, we need to define the function $W:[e,\infty) \rightarrow [1,\infty)$ implicitly by
$$
\frac{\exp(W(x))}{W(x)}=x.
$$
This function is linked to the Lambert $\mathcal{W}$ function by the relation $W(x)=-\mathcal{W}\Bigl(\frac{-1}{x}\Bigr)$ in which we take the solution larger than 1.

\begin{corollary}\label{cor:1}
Let $\theta \in (0,1]$ be fixed and write
$$
\alpha:=W\biggl(\frac{e}{\theta\log 2}\biggr).
$$
Let also $n \ge 2$ be a fixed squarefree integer. Then, assuming that $\alpha - 1 < \log \omega(n)$ and that $\omega(n) \ge 56$, we have
\begin{equation}\label{cor:1:res}
\mathcal{H}_{\theta}(n) \le n\exp{\scriptstyle\Bigl(-\frac{\theta\log 2}{\alpha}\omega(n)\log \omega(n)+\bigl(\frac{\alpha-1}{1-\frac{\alpha-1}{\log \omega(n)}}\bigr)^3\frac{\omega(n)}{\exp\bigl(\frac{\alpha-1}{1+\frac{\alpha-1}{\log \omega(n)}}\bigr)\log \omega(n)}\Bigr)}.
\end{equation}
\end{corollary}
We record some approximate values of $\alpha=\alpha(\theta)$ in Table 1.
\begin{center}
\begin{tabular}{|c||c|c|c|c|c|c|c|c|c|c|}\hline
$\theta$ & 0.1 & 0.2 & 0.3 & 0.4 & 0.5 & 0.6 & 0.7 & 0.8 & 0.9 & 1 \\ \hline
$\alpha(\theta)$ & 5.34 & 4.47 & 3.94 & 3.54 & 3.23 & 2.96 & 2.72 & 2.50 & 2.30 & 2.11\\ \hline
\end{tabular}
\vskip 5pt
{\sc Table 1}
\end{center}

In \cite{sr}, it has been shown that the number of divisors function, noted $\tau(\cdot)$, satisfies the inequality
$$
\tau(n) \le \biggl(\frac{\log (n\gamma(n))}{\omega(n)}\biggr)^{\omega(n)}\beta(n) \quad (n \ge 2)
$$
where
$$
\beta(n):=\prod_{p \mid n}\frac{1}{\log p}.
$$
This inequality has been extensively worked out in the author's Ph. D. Thesis \cite{jmdk:pl}. It is worth mentioning that the function $\beta(n)$ is intimately linked to the value of $\tau(n)$ in more than one way. In particular, it follows from Theorem $5.3$ p.491 of \cite{gt} that
$$
\frac{(\log z)^{\omega(n)}}{\omega(n)!}\beta(n) \le |\{1 \le j \le z:\ \gamma(j) \mid \gamma(n)\}| \le \frac{(\log (z\gamma(n)))^{\omega(n)}}{\omega(n)!}\beta(n)
$$
so that
$$
\tau(n,z):=\sum_{\substack{d \mid n \\ d \le z}}1 \le \frac{(\log (z\gamma(n)))^{\omega(n)}}{\omega(n)!}\beta(n).
$$
In the special case where $n$ is squarefree, one prefers the estimate
$$
\tau(n,z) \le \sum_{0 \le j \le D(n,z)}\binom{\omega(n)}{j}
$$
where
$$
D(n,z) := \max_{\substack{d \mid n \\ d \le z}}\omega(d).
$$
For comparison, the argument in \cite{pe} allows one to establish that
$$
-\max_{\substack{0 \le j \le D(n,z) \\ 2 \nmid j}}\binom{\omega(n)-1}{j} \le M(n,z) \le \max_{\substack{0 \le j \le D(n,z) \\ 2 \mid j}}\binom{\omega(n)-1}{j}
$$
for every integer $n \ge 2$.

Let $s \ge 1$ be a fixed integer. For any integers $n \ge 1$ we write $d(n):=\{d: d \mid n\}$. We define the $s$-th multiplicative energy of $n$ to be
$$
E_s(n):=|\{(d_1,\dots,d_s,d_{s+1},\dots,d_{2s}) \in d(n)^{2s}:\ d_1 \cdots d_s = d_{s+1} \cdots d_{2s}\}|.
$$
In particular, we trivially have $E_s(n) \le \tau(n)^{2s-1}$. In what follows, $A(i,j)$ are Eulerian numbers of the first kind that can be computed by using the formula
$$
A(i,j) = \sum_{v = 0}^{j}\binom{i+1}{v}(-1)^v(j+1-v)^i \quad (0 \le j \le i-1,\ i,j \in \mathbb{Z}).
$$

\begin{theorem}\label{thm:3}
Let $s,n \ge 2$ be positive integers. Then the inequality
\begin{equation}\label{thm:3:result}
\tau(n)^{2s-1}\left(\frac{A(2s-1,s-1)}{(2s-1)!}\right)^{\omega(n)} < E_s(n) \le \tau(n)^{2s-1}\left(\frac{1}{2^{2s-1}}\binom{2s}{s}\right)^{\omega(n)}
\end{equation}
holds.
\end{theorem}

\begin{remark}
It is possible to establish that
$$
\frac{A(2s-1,s-1)}{(2s-1)!} \sim \sqrt{\frac{3}{\pi s}} \quad \mbox{and} \quad \frac{1}{2^{2s-1}}\binom{2s}{s} \sim \sqrt{\frac{4}{\pi s}} \quad (s \rightarrow \infty).
$$
The first relation is deduced from the identity
\begin{equation}\label{int:id}
\int_{-\infty}^{\infty}\biggl(\frac{\sin(x)}{x}\biggr)^{2s}dx=\pi\frac{A(2s-1,s-1)}{(2s-1)!} \quad (s \in \mathbb{N}).
\end{equation}
\end{remark}

The upper bound in \eqref{thm:3:result} is in fact an equality in the case where $n$ is squarefree. In this direction, we will see that the proof gives a much more general result.

Throughout the paper, we denote the $k$-th prime number by $p_k$. Also, for each $k \ge 0$, we denote by $n_k$ the number $\prod_{j=1}^{k}p_j$ (so that $n_0=1$).

{\bf Acknowledgment.}  I thank Jean-Marie De Koninck for his interest in this article and Thomas J. Ransford for a discussion about the integrals \eqref{int:id} many years ago.

\section{Preliminary lemmas}

\begin{lemma}\label{surj}
Let $i \ge 0$ and $j \ge 1$ be integers. Let also denote the number of surjections from a set of $i$ elements to a set of $j$ elements by $S_{i,j}$. It satisfies
$$
S_{i,j}=\sum_{v=0}^{j}(-1)^{j-v}\binom{j}{v}v^i.
$$
\end{lemma}

\begin{proof}
It is a well known result. We remark that it implies that
$$
\sum_{v=0}^{j}(-1)^{j-v}\binom{j}{v}v^i=\left\{\begin{array}{ll} 0 & \mbox{for}\ i=0,\dots,j-1, \\ i! & \mbox{for}\ i=j. \end{array}\right.
$$
\end{proof}

\begin{lemma}\label{Mitchell}
Let $0 \le u_1 < \cdots < u_{\ell}$ and $0 < x_1 < \cdots < x_{\ell}$ be two sequences of real numbers. Then the generalized Vandermonde determinant satisfies
$$
\begin{vmatrix} x_1^{u_1} & x_1^{u_2} & \cdots & x_1^{u_{\ell}} \\ x_2^{u_1} & x_2^{u_2} & \cdots & x_2^{u_{\ell}} \\ \vdots & \vdots & \cdots & \vdots \\ x_{\ell}^{u_1} & x_{\ell}^{u_2} & \cdots & x_{\ell}^{u_{\ell}} \end{vmatrix} > 0.
$$
\end{lemma}

\begin{proof}
This is known as a result of Mitchell \cite{ohm}. A modern proof uses the Lemma $A2$ of \cite{mnh}.
\end{proof}

For $x \in \mathbb{R}$, we define the sign function by
$$
\sgn{x}:=\left\{\begin{array}{ll}-1 & \mbox{if}\ x < 0, \\ 0 & \mbox{if}\ x = 0, \\ 1 & \mbox{otherwise}. \end{array}\right.
$$

\begin{lemma}\label{sign}
Let $\lambda \ge 1$ and $m \ge 0$ be integers. Let also $F_{\lambda,m}(x)$ be the polynomial of minimal degree that satisfies
$$
F_{\lambda,m}(j):=\sgn{j}j^m \quad (j \in \{-\lambda,\dots,\lambda\}).
$$
We assume that $0 \le m \le 2\lambda-1$. We have
$$
\left\{\begin{array}{l}
\mbox{$F_{\lambda,m}(x)$ is an odd function of degree $2\lambda-1$ with a leading term}\\
\mbox{of sign $(-1)^{\frac{2\lambda-m-2}{2}}$ if $m$ is even,}\\
\mbox{$F_{\lambda,m}(x)$ in an even function of degree $2\lambda$ with a leading term}\\
\mbox{of sign $(-1)^{\frac{2\lambda-m-1}{2}}$ if $m$ is odd}.
\end{array}\right.
$$
\end{lemma}

\begin{proof}
We first assume that $m \ge 1$ is odd. From Lagrange interpolation with $2\lambda+1$ points, we have $\deg F_{\lambda,m}(x) \le 2\lambda$. Now, the polynomial
$$
G_{\lambda,m}(x):=F_{\lambda,m}(x)-F_{\lambda,m}(-x)
$$
has at least $2\lambda+1$ roots, so that $G_{\lambda,m}(x)$ is identically $0$ and we deduce that $F_{\lambda,m}(x)$ is an even function. Therefore, we search for a polynomial of the type
$$
F_{\lambda,m}(x):=\sum_{j=1}^{\lambda}a_{2j}x^{2j}.
$$
We get to the linear system
$$
\begin{pmatrix} 1^{2} & 1^{4} & \cdots & 1^{2\lambda} \\ 2^{2} & 2^{4} & \cdots & 2^{2\lambda} \\ \vdots & \vdots & \cdots & \vdots \\ \lambda^{2} & \lambda^{4} & \cdots & \lambda^{2\lambda} \end{pmatrix}\begin{pmatrix} a_2 \\ a_4 \\ \vdots \\ a_{2\lambda} \end{pmatrix} = \begin{pmatrix} 1^m \\ 2^m \\ \vdots \\ \lambda^m \end{pmatrix}.
$$
By Cramer's rule,
$$
\begin{vmatrix} 1^{2} & 1^{4} & \cdots & 1^{2\lambda-2} & 1^{m} \\ 2^{2} & 2^{4} & \cdots & 2^{2\lambda-2} & 2^{m} \\ \vdots & \vdots & \cdots & \vdots & \vdots \\ \lambda^{2} & \lambda^{4} & \cdots & \lambda^{2\lambda-2} & \lambda^{m} \end{vmatrix}=a_{2\lambda}\begin{vmatrix} 1^{2} & 1^{4} & \cdots & 1^{2\lambda} \\ 2^{2} & 2^{4} & \cdots & 2^{2\lambda} \\ \vdots & \vdots & \cdots & \vdots \\ \lambda^{2} & \lambda^{4} & \cdots & \lambda^{2\lambda} \end{vmatrix}
$$
so that we deduce from Lemma \ref{Mitchell} that $\sgn{a_{2\lambda}}=(-1)^{\frac{2\lambda-m-1}{2}}$. The result follows from the fact that there is a unique such interpolating polynomial of degree at most  $2\lambda$.

In the case where $m \ge 0$ is even, we simply observe that $F_{\lambda,m}(x):=\frac{F_{\lambda,m+1}(x)}{x}$. The proof is complete.
\end{proof}

Let's define
$$
\eta(n,t):=\prod_{p \mid n}\biggl(1+\frac{1}{p^{\frac{1}{t}}}\biggr).
$$

\begin{lemma}\label{comp}
Let $t \ge 2$ and $k \ge 1$ be positive integers. Then
\begin{equation}\label{c-easy}
\eta(n_k,t) < \left\{\begin{array}{ll} \exp(k^{1-1/t}) & \mbox{if}\ t = 2\ \mbox{and}\ k \le 55,\\ \exp\bigl(k^{1-1/t}-\frac{\log t}{t}\bigr) & \mbox{otherwise}.\end{array}\right.
\end{equation}
Also,
\begin{equation}\label{c-hard}
\eta(n_k,t) \le \exp\biggl(C\frac{k^{1-1/t}}{(1-1/t)\log^{1/t}_+ k}-(1-\delta_{2,t})\frac{\log t}{t}\biggr)
\end{equation}
with $C=1.07073472\dots$. The constant $C$ is best possible and is attained only at $t=2$ and $k=2149$.
\end{lemma}

\begin{proof}
We will begin with the proof of \eqref{c-easy}. We will prove this result by induction for every single values of $t \ge 2$. For $t=2,\dots,99$ we verify with a computer for every value of $k=1,\dots,56$. For $t \ge 100$ there is no need to verify since
\begin{equation}\label{ln2}
\frac{\log t}{t}+\sum_{j=1}^{k}\log\biggl(1+\frac{1}{p_j^{1/t}}\biggr) \le \frac{\log 100}{100} + k\log 2 < 0.74k
\end{equation}
while $k^{1-1/t} \ge 0.74k$ for each $k=1,\dots,56$. We consider $t \ge 2$ as fixed. Now, for a fixed $k \ge 57$, we assume that the result holds for $k-1$. We will establish that
\begin{equation}\label{ind-1}
(k-1)^{1-1/t}+\log\biggl(1+\frac{1}{p_k^{1/t}}\biggr) < k^{1-1/t}
\end{equation}
which is clearly enough for the induction step with this value of $t$. We see that \eqref{ind-1} holds if
\begin{eqnarray}\nonumber
\frac{1}{p_k^{1/t}} & < & k^{1-1/t}-(k-1)^{1-1/t} \\ \nonumber
& \Uparrow &\\ \label{ind-2}
\frac{1}{p_k^{1/t}} & < & \frac{1-1/t}{k^{1/t}}
\end{eqnarray}
from the mean value theorem. Now, it is known that $p_k > k\log k$ for each $k\ge 1$, see \cite{jbr}. Using this inequality, we have that \eqref{ind-2} holds if
\begin{eqnarray*}
\frac{1}{(k\log k)^{1/t}}  <  \frac{1-1/t}{k^{1/t}} & \Leftrightarrow & \frac{1}{\log k} < \left(1-\frac{1}{t}\right)^t \\
& \Leftarrow & \frac{1}{\log k} < \frac{1}{4} \\
& \Leftarrow & k \ge 57.
\end{eqnarray*}
We have used the fact that the function $\left(1-\frac{1}{x}\right)^x $ is strictly increasing for $x > 1$. This concludes the proof of inequality \eqref{ind-1} and thus the induction step for the fixed value of $t$. Inequality \eqref{c-easy} is established.

We now turn to the proof of \eqref{c-hard}. The argument is very similar, that is we proceed by induction for every single value of $t \ge 2$. For $t=2,\dots,99$ we verify with a computer for each value of $k$ from 1 to what is written in Table 2.
\begin{center}
\begin{tabular}{|c||c|c|c|c|c|c|c|c|}\hline
$t$ & 2 & 3 & 4 & 5 & 6 & 7 & 8\ \mbox{to}\ 99\\ \hline
$k$ & 3750230 & 1936 & 155 & 44 & 20 & 12 & 8\\ \hline
\end{tabular}
\vskip 5pt
{\sc Table 2}
\end{center}
For each $t \ge 100$, by using \eqref{ln2}, it is enough to have
$$
0.74k < \frac{C}{1-1/t}\frac{k^{1-1/t}}{\log^{1/t}_+ k} \quad \Leftrightarrow \quad (k\log_+ k)^{1/t} < \frac{C}{(1-1/t) \cdot 0.74}
$$
which is easily seen to hold for $k=1,\dots,8$.

Let's consider $t \ge 2$ as fixed. We assume that
\begin{equation}\label{c-hard-ind}
\sum_{j=1}^{J}\log\biggl(1+\frac{1}{p_j^{1/t}}\biggr) < C\frac{J^{1-1/t}}{(1-1/t)\log^{1/t}J}-(1-\delta_{2,t})\frac{\log t}{t}
\end{equation}
holds at $J=k-1$ and we want to show that it holds with $J=k(\ge 9)$ . It is enough to show that
\begin{eqnarray}\nonumber
\frac{1}{p_k^{1/t}} & < & C\frac{k^{1-1/t}}{(1-1/t)\log^{1/t} k}-C\frac{(k-1)^{1-1/t}}{(1-1/t)\log^{1/t}(k-1)} \\ \nonumber
& \Uparrow & \\ \label{ind-3}
\frac{1}{p_k^{1/t}} & < & \frac{C}{1-1/t}\biggl(\frac{1-1/t}{(k\log k)^{1/t}}-\frac{1/t}{k^{1/t}\log^{1+1/t} k}\biggr)
\end{eqnarray}
from the mean value theorem and the fact that $f_t(x):=\frac{x^{1-1/t}}{\log^{1/t} x}$ satisfies $f''_t(x) < 0$ if $x \ge 6 > \exp\Bigl(\frac{2/t-1+\sqrt{5-4/t}}{2(1-1/t)}\Bigr)$ for each $t \ge 2$. Again, by using $p_k > k\log k$ for each $k \ge 1$, we deduce that \eqref{ind-3} holds if
\begin{eqnarray*}
\frac{1}{(k\log k)^{1/t}} & < & \frac{C}{1-1/t}\biggl(\frac{1-1/t}{(k\log k)^{1/t}}-\frac{1/t}{k^{1/t}\log^{1+1/t} k}\biggr) \\
& \Updownarrow & \\
1 & < & C\biggl(1-\frac{1}{(t-1)\log k}\biggr) \\
& \Updownarrow & \\
\log k & > & \frac{C}{(t-1)(C-1)}
\end{eqnarray*}
which holds for $k$ greater that the corresponding value in Table 2 if $t=2,\dots,99$ or for $k \ge 9$ if $t \ge 100$. This completes the inductive step for the fixed value of $t \ge 2$ and the proof is complete.
\end{proof}

\section{Proof of Theorem \ref{thm:1} and \ref{thm:2}}

We write
\begin{eqnarray*}
L_t(n) & = & \int_{1}^{n}\sum_{\substack{ d_1, \dots, d_t \mid n \\ d_1, \dots, d_t \le z}}\mu(d_1) \cdots \mu(d_t)dz \\
& = & \sum_{d_1, \dots, d_t \mid n}\mu(d_1) \cdots \mu(d_t)\int_{1}^{n}\chi(d_1,z)\cdots\chi(d_t,z)dz \\
& = & \sum_{d_1, \dots, d_t \mid n}\mu(d_1) \cdots \mu(d_t)(n-\max(d_1,\dots,d_t))
\end{eqnarray*}
where
$$
\chi(d,z):=\left\{\begin{array}{ll} 0 & \mbox{if}\ d < z, \\ 1 & \mbox{otherwise}, \end{array}\right. \quad(d \in \mathbb{N}, z \in \mathbb{R}).
$$
Now, for $n \ge 2$, we rearrange the terms according to the number $j=1,\dots,t$ of $d_i$ that are maximal at the same time and we use the fact that $M(n,n)=0$ to get to

\begin{eqnarray}\nonumber
L_t(n) & = & -\sum_{d_1, \dots, d_t \mid n}\mu(d_1) \cdots \mu(d_t)\max(d_1,\dots,d_t)\\ \nonumber
& = & -\sum_{j=1}^{t}\binom{t}{j}\sum_{d \mid n}\mu(d)^j d\sum_{\substack{d_1, \dots, d_{t-j} \mid n \\ d_1, \dots, d_{t-j} < d}}\mu(d_1) \cdots \mu(d_{t-j})\\ \nonumber
& = & -\sum_{j=1}^{t}\binom{t}{j}\sum_{d \mid n}\mu(d)^j d\bigg(\sum_{\substack{e \mid n \\ e < d}}\mu(e)\bigg)^{t-j}\\ \nonumber
& = & -\sum_{d \mid n}d\bigg(\bigg(\sum_{\substack{e \mid n \\ e \le d}}\mu(e)\bigg)^{t}-\bigg(\sum_{\substack{e \mid n \\ e < d}}\mu(e)\bigg)^{t}\bigg)\\ \label{relimp}
& = & (-1)^{t}n\sum_{d \mid n}\frac{1}{d}\bigg(\bigg(\sum_{\substack{e \mid n \\ e \ge n/d}}\mu(e)\bigg)^{t}-\bigg(\sum_{\substack{e \mid n \\ e > n/d}}\mu(e)\bigg)^{t}\bigg).
\end{eqnarray}
Thus, let $1 = d_1 < d_2 < \cdots < d_{2^{\omega(n)}} = n$ be the sequence of divisors of $n$. We write
$$
\mathcal{J}_\rho(n) := \sum_{i=1}^{2^{\omega(n)}}\frac{i^{\rho}}{d_i} \qquad (\rho \in \mathbb{R}_{\ge 0}).
$$
Now, we deduce from \eqref{relimp} that
\begin{eqnarray*}
|L_t(n)| & \le & n\sum_{i=1}^{2^{\omega(n)}}\frac{i^t-(i-1)^t}{d_i}\\
& \le & tn\mathcal{J}_{t-1}(n).
\end{eqnarray*}
Also, for any integer value of $\rho \ge 0$ and $\sigma \in \mathbb{R}_{\ge 0}$, we can write
\begin{eqnarray*}
\mathcal{J}_\rho(n) & = & \sum_{d \mid n}\frac{\tau^\rho(n,d)}{d} \le \sum_{d \mid n}\frac{1}{d}\biggl(\sum_{e \mid n} \left(\frac{d}{e}\right)^{\sigma}\biggr)^\rho\\
& = & \sum_{d \mid n}d^{\rho\sigma-1}\biggl(\sum_{e \mid n} \frac{1}{e^\sigma}\biggr)^\rho = \prod_{p \mid n}(1+p^{\rho\sigma-1})\left(1+\frac{1}{p^\sigma}\right)^\rho\\
& = & \prod_{p \mid n}\biggl(1+\frac{1}{p^{\frac{1}{\rho+1}}}\biggr)^{\rho+1} = \eta^{\rho+1}(n,\rho+1)
\end{eqnarray*}
where we have used $\sigma=\frac{1}{\rho+1}$. We thus get to
$$
|L_t(n)| \le tn\eta^t(n,t) \le tn\eta^t(n_{\omega(n)},t).
$$
The results then follow from Lemma \ref{comp}. The proof is complete.

\begin{remark}
The function $\mathcal{J}_\rho(\cdot)$ satisfies
$$
\mathcal{J}_\rho(n) \le \mathcal{J}_\rho(n_{\omega(n)}) \qquad (n \in \mathbb{N}).
$$
Indeed, let $n=q_1\cdots q_{\omega(n)}$ with $q_1 < \cdots < q_{\omega(n)}$ the factorization of $n$. Thus, since $p_{r_1}\cdots p_{r_l} \le q_{r_1}\cdots q_{r_l}$, it follows that the $i$-th term  in the ordered sequence of divisors of $n_{\omega(n)}$ is at most equal to the $i$-th term in the corresponding sequence for $n$.
\end{remark}

\section{Proof of Corollary \ref{cor:1}}

Since the function $M(n,z)$ is constant for $z \in [j,j+1)$ ($j \in \mathbb{Z}_{\ge 0}$) and $M(n,n)=0$ for $n \ge 2$, we deduce that
$$
\mathcal{H}_{\theta}(n) \le \frac{L_t(n)}{2^{t\theta\omega(n)}} \qquad (\mbox{if}\ 2 \mid t).
$$
From Theorem \ref{thm:2} and the hypothesis $\omega(n) \ge 56$, we have
$$
\mathcal{H}_{\theta}(n) \le n\exp\biggl(t\omega(n)\biggl(\frac{1}{\omega(n)^{1/t}}-\theta\log 2\biggr)\biggr).
$$
Now, the idea is simply to optimize this last inequality over the even integers $t \ge 2$. Our strategy is to find the exact value $t_0 \in (1,\infty)$ and to estimate the variation caused by $t=t_0+\xi$ with $|\xi| \le 1$. We write
$$
f(x):=x\biggl(\theta\log 2-\frac{1}{\omega(n)^{1/x}}\biggr),
$$
so that
$$
f'(x)=\theta\log 2-\frac{1}{\omega(n)^{1/x}}-\frac{\log \omega(n)}{x\omega(n)^{1/x}} \quad \mbox{and} \quad f''(x)=-\frac{\log^2\omega(n)}{x^3\omega(n)^{1/x}}.
$$
Let's write $t_0=c\log \omega(n)$. We have $f'(t_0)=0$ if and only if
\begin{eqnarray*}
\omega(n)^{1/t_0}\theta\log 2 = 1 + \frac{\log \omega(n)}{t_0} & \Leftrightarrow & \exp(1/c)\theta\log 2 = 1 + \frac{1}{c}\\
& \Leftrightarrow & \frac{\exp(1+1/c)}{1+1/c} = \frac{e}{\theta\log 2}\\
& \Leftrightarrow & 1+\frac{1}{c} = \alpha = W\biggl(\frac{e}{\theta\log 2}\biggr)
\end{eqnarray*}
so that $t_0=\frac{\log \omega(n)}{\alpha-1}$ which is strictly larger than 1 by hypothesis. We verify that $f(t_0)=\frac{\theta\log 2}{\alpha}\log \omega(n)$. Now, we have
\begin{eqnarray*}
|f(t)-f(t_0)| & \le & \sup_{z \in (t_0-1,t_0+1)}|f'(z)| = \sup_{z \in (t_0-1,t_0+1)}|f'(z)-f'(t_0)|\\
& \le & \sup_{\zeta \in (t_0-1,t_0+1)}|f''(\zeta)|
\end{eqnarray*}
from the mean value theorem applied twice. The result follows from the estimate
$$
\sup_{\zeta \in (t_0-1,t_0+1)}|f''(\zeta)| < \frac{\log^2\omega(n)}{(t_0-1)^3\omega(n)^{1/(t_0+1)}}
$$
which holds since $t_0 > 1$. The proof is complete.

\section{Proof of Theorem \ref{thm:3}}

We assume throughout the proof that $s \ge 2$ is a fixed integer. The function $E_s(n)$ is multiplicative, so it will be enough to show that
{\small$$
\tau(p^\alpha)^{2s-1}\left(\frac{A(2s-1,s-1)}{(2s-1)!}\right) < E_s(p^\alpha) \le \tau(p^\alpha)^{2s-1}\left(\frac{1}{2^{2s-1}}\binom{2s}{s}\right)\quad(\alpha \ge 1)
$$}\par
\noindent for any prime $p$.

Now, for a fixed prime $p$, the function $E_s(p^\alpha)$ counts the number of solutions to the system
$$
R_s(\alpha):=|\{(\alpha_1,\dots,\alpha_{2s}) \in \{0,\dots,\alpha\}^{2s}:\ \alpha_1+\cdots+\alpha_s = \alpha_{s+1}+\cdots+\alpha_{2s}\}|.
$$
We clearly have $R_s(0)=1$ and also
$$
R_s(1):=\sum_{j=0}^{s}\binom{s}{j}^2=\binom{2s}{s},
$$
an identity that follows from $(x+1)^{2s}=(x+1)^s(x+1)^s$.
In general, $R_s(\alpha)$ is the coefficient of $x^0$ in the expansion of
\begin{eqnarray*}
N_{s,\alpha}(x) & := & \biggl((1+x+\cdots+x^{\alpha})\left(1+\frac{1}{x}+\cdots+\frac{1}{x^\alpha}\right)\biggr)^s\\
& = & \left(\frac{x^{\alpha+2}+x^{-\alpha}-2x}{(1-x)^2}\right)^s\\
& = & (x^{\alpha+2}+x^{-\alpha}-2x)^s\sum_{j \ge 0}\binom{2s-1+j}{2s-1}x^j
\end{eqnarray*}
from which we deduce that
{\small\begin{eqnarray*}
R_s(\alpha) & = & \sum_{\substack{a+b+c=s \\ (\alpha+2)a-\alpha b+c \le 0}}\binom{s}{a,b,c}(-2)^{c}\binom{2s-1-(\alpha+2)a+\alpha b-c}{2s-1}\\
& = & \sum_{i=0}^{s}\sum_{\substack{0 \le j \le s-i \\ (\alpha+1)(i-j)+s \le 0}}\binom{s}{i,j,s-i-j}(-2)^{s-i-j}\binom{(\alpha+1)(j-i)+s-1}{2s-1}\\
& = & \sum_{i=0}^{s}\sum_{\substack{-i \le v \le s \\ -(\alpha+1)v+s \le 0}}\binom{s}{i,v+i,s-v-2i}(-2)^{s-v-2i}\binom{(\alpha+1)v+s-1}{2s-1}\\
& = & \sum_{v = 1}^{s}(-1)^{s-v}\binom{(\alpha+1)v+s-1}{2s-1}\sum_{i=0}^{s}\binom{s}{i,v+i,s-v-2i}2^{s-v-2i}\\
& = & \sum_{v = 1}^{s}(-1)^{s-v}\binom{2s}{s-v}\binom{(\alpha+1)v+s-1}{2s-1} =: P_s(\alpha+1).
\end{eqnarray*}}\par
\noindent The last expression follows from
$$
\sum_{i=0}^{s}\binom{s}{i,v+i,s-v-2i}2^{s-v-2i}=\binom{2s}{s-v}
$$
that can be shown by using the identity $(x+1)^{2s}=((x^2+2x)+1)^{s}$.

Now, the idea of the proof is to show that $P_s(x)$ is an odd function with strictly positive coefficients (of $x^j$ with $j$ odd) so that it is clear that the function $\frac{P_s(x)}{x^{2s-1}}$ has a strictly negative derivative. With this in mind, we write
\begin{eqnarray*}
P_s(x) & = & \sum_{v = 1}^{s}(-1)^{s-v}\binom{2s}{s-v}\binom{xv+s-1}{2s-1}\\
& = & \frac{x}{(2s-1)!}\sum_{v = 1}^{s}(-1)^{s-v}\binom{2s}{s-v}v\prod_{j=1}^{s-1}(v^2x^2-j^2)
\end{eqnarray*}
so that we turn to
\begin{eqnarray*}
\frac{(2s-1)!P_s(ix)}{ix} & = & \sum_{v = 1}^{s}(-1)^{v-1}\binom{2s}{s-v}v\prod_{j=1}^{s-1}(v^2x^2+j^2)\\
& = & \sum_{v = 1}^{s}(-1)^{v-1}\binom{2s}{s-v}v\sum_{r=0}^{s-1}b_{2r}v^{2r}x^{2r}\\
& = & \sum_{r=0}^{s-1}b_{2r}x^{2r}\sum_{v = 1}^{s}(-1)^{v-1}\binom{2s}{s-v}v^{2r+1}
\end{eqnarray*}
where each $b_{2r}$ with $r=0,\dots,s-1$ is a strictly positive coefficient. By writing
$$
c_r:=\sum_{v = 1}^{s}(-1)^{v-1}\binom{2s}{s-v}v^{2r+1},
$$
we deduce that it is enough to show that $\sgn{c_r}=(-1)^{r}$. We write
\begin{eqnarray*}
c_r & = & -\sum_{v = 0}^{s}(-1)^{s-v}\binom{2s}{v}(s-v)^{2r+1}\\
& = & (-1)^{s+1}\sum_{v = 0}^{2s}(-1)^{2s-v}\binom{2s}{v}Q_{s,2r+1}(v)
\end{eqnarray*}
where $Q_{s,2r+1}(x)$ is the Lagrange polynomial of degree at most $2s$ for which
$$
Q_{s,2r+1}(v) = \left\{\begin{array}{ll} (s-v)^{2r+1} & v=0,\dots,s,\\ 0 & v=s+1,\dots,2s. \end{array}\right.
$$
From Lemma \ref{surj} and the remark in the proof, we deduce that $\sgn{c_r}=(-1)^{s+1}\sgn{e_{2s}}$ where $e_{2s}$ is the leading term of $Q_{s,2r+1}(x)$. Now, since
$$
Q_{s,2r+1}(x+s)+Q_{s,2r+1}(-x+s)=F_{s,2r+1}(x)
$$
the function in Lemma \ref{sign}. We deduce that $\sgn{2e_{2s}}=\sgn{e_{2s}}=(-1)^{s-r-1}$ so that $\sgn{c_r}=(-1)^r$ as wanted. The proof is complete.

\section{Concluding remark}

Let's consider the quantity
$$
T_s(\alpha):=|\{(\alpha_1,\dots,\alpha_{s}) \in \{-\alpha,\dots,\alpha\}^{s}:\ \alpha_1+\cdots+\alpha_s=0\}|.
$$
The methods used in the proof of Theorem \ref{thm:3} also apply to $T_s(\alpha)$. That is, the function $\frac{T_s(\alpha)}{(2\alpha+1)^{s-1}}$ is strictly decreasing for integer values of $\alpha \ge 0$ when $s \ge 3$, it is constant for $s=1$ or 2.

{\sc D\'epartement de math\'ematiques et de statistique, Universit\'e Laval, Pavillon Alexandre-Vachon, 1045 Avenue de la M\'edecine, Qu\'ebec, QC G1V 0A6} \\
{\it E-mail address:} {\tt Patrick.Letendre.1@ulaval.ca}


\begin{thebibliography}{99}

\normalsize
\baselineskip=17pt


\bibitem{ngdb:cvet:dk} N. G. De Bruijn, C. van E. Tengbergen and D. Kruyswijk, \emph{On the set of divisors of a number}, Nieuw Arch. Wisk. 23 (1951), 191--193. 

\bibitem{jmdk:pl} J.-M. De Koninck and P. Letendre, \emph{New upper bounds for the number of divisors function}, P. Letendre's Ph. D. Thesis Chapter 3.

\bibitem{rdlb:gt} R. de la Bret\`eche and G. Tenenbaum, \emph{Oscillations localis\'ees sur les diviseurs}, J. Lond. Math. Soc. (2) 85 (2012), no. 3, 669--693. 

\bibitem{pe} P. Erd\H{o}s, \emph{On a problem in elementary number theory}, Math. Student 17 (1949), 32--33.

\bibitem{pe:rrh} P. Erd\H{o}s and R. R. Hall, \emph{On the M\"obius function}, J. Reine Angew. Math. 315 (1980), 121--126.

\bibitem{pe:ik} P. Erd\H{o}s and I. K\'atai, \emph{Non complete sums of multiplicative functions}, Period. Math. Hungar. 1 1971 no. 3, 209--212. 

\bibitem{rrh} R. R. Hall, \emph{A problem of Erd\H{o}s and K\'atai}, Mathematika 21 (1974), 110--113.

\bibitem{mnh} M. N. Huxley, \emph{The integer points in a plane curve}, Funct. Approx. Comment. Math. 37 (2007), part 1, 213--231.

\bibitem{hm} H. Maier, \emph{On the M\"obius function}, Trans. Amer. Math. Soc. 301 (1987), no. 2, 649--664. 

\bibitem{ohm} O. H. Mitchell, \emph{Note on determinants of powers}, Amer. J. Math. 4 (1881), 341--344.

\bibitem{sr} S. Ramanujan, \emph{Highly composite numbers}, Proc. London Math. Soc. (2) 14 (1915), 347--409.

\bibitem {jbr} J. B. Rosser, \emph{The $n$-th prime is greater than $n\log n$}, Proc Lond. Math. Soc. (2), vol. 45 (1939), 21--44.

\bibitem{gt} G. Tenenbaum, \emph{Introduction \`a la th\'eorie analytique et probabiliste des nombres}, $3^e$ \'edition, Belin, 2008.


\end{thebibliography}
\end{document}